\documentclass[12pt,reqno]{amsart}
\usepackage{amsmath, amsthm, amssymb}
\usepackage[shortlabels]{enumitem}
\usepackage[numbers]{natbib}
\usepackage{mathtools}
\usepackage[numbers]{natbib}
\usepackage[margin=1in]{geometry}
\usepackage{graphicx}
\usepackage[bookmarks,colorlinks,breaklinks]{hyperref}  
\usepackage{cleveref}
\usepackage{color}
\hypersetup{linkcolor=red,citecolor=blue,filecolor=dullmagenta,urlcolor=darkblue} 
\makeatletter
\def\thm@space@setup{%
  \thm@preskip=0.5em\thm@postskip=\thm@preskip%
}
\makeatother

\newtheoremstyle{named}{}{}{\\itshape}{}{\bfseries}{.}{.5em}{\thmnote{#3's }#1}
\theoremstyle{named}

\theoremstyle{plain}
\newtheorem{thm}{Theorem}[section]
\newtheorem{definition}[thm]{Definition}
\newtheorem{proposition}[thm]{Proposition}

\newtheorem{lem}[thm]{Lemma}
\newtheorem{cor}[thm]{Corollary}
\newtheorem{prop}[thm]{Proposition}

\crefformat{footnote}{#2\footnotemark[#1]#3}
\newcommand{\conj}[1]{\overline{#1}}
\newcommand{\tr}{\operatorname{tr}}

\newcommand{\roi}[1]{\mathcal{O}_{#1}}

\newcommand{\prim}{\mathfrak{p}}

\newcommand{\val}{\mathrm{val}}

\newcommand\restr[2]{{
  \left.\kern-\nulldelimiterspace 
  #1 
  \vphantom{\big|} 
  \right|_{#2} 
  }}

\newcommand{\catname}[1]{{\textsf{{#1}}}}
\newcommand{\Herm}{\catname{PtdHerm}}
\newcommand{\Proj}{\catname{PtdProjHerm}}
\newcommand{\Quat}{\catname{Quat}}
\newcommand{\Ord}{\catname{Ord}}

\begin{document}
\author{Gordan Savin}
\address{Department of Mathematics, Salt Lake City, UT 84112}
\email[Gordan Savin]{savin@math.utah.edu}
\author{Michael Zhao}
\address{Department of Mathematics, Salt Lake City, UT 84112}
\email[Michael Zhao]{mzhaox@gmail.com}
\title{Binary Hermitian Forms and Optimal Embeddings}
\date{\today}
\maketitle

\begin{abstract} Fix a quadratic order over the ring of integers. 
An embedding of the quadratic order into  a quaternionic order naturally gives an integral binary hermitian form over the quadratic order. We show that, in certain cases, this 
correspondence is 
a discriminant preserving bijection between the isomorphism classes of embeddings and integral binary hermitian forms. 
\end{abstract} 
\section{Introduction}

Let $K$ be a field of any characteristic and fix $L$, a separable quadratic extension of $K$. Let $n_L : L \rightarrow K$, $n_L(l)=l\bar l$,  be the norm map.
In this paper we study relationship between embeddings of $L$ into a quaternion algebra and vector spaces over $L$ with a hermitian form, 
as well as an integral version of this problem. 
Let $V$ be a vector space over $L$.  A hermitian form on $V$ is a 
quadratic form $h: V \rightarrow K$ such that $h(lv) = n_L(l) h(v)$ for all $l\in L$ and $v\in V$. This definition is different than what is customary in 
the literature, where a hermitian symmetric form is a function $s : V \times V \rightarrow L$, linear in the first variable, 
and $s(y,x)= \conj{s(x,y)}$ for all $x,y\in V$. However, given $h$, we show that there exists a unique $s$ 
 such that $h(v)=s(v,v)$ for all $v\in V$, even if the characteristic of $K$ is 2.

 Let $i: L \rightarrow H$ be an embedding of $L$ into a quaternion algebra $H$. Let $n: H \rightarrow K$ be the (reduced) norm. 
 Then $H$ is naturally a vector space over $L$ of dimension 2,  with a hermitian form $h: = n$  representing 1, since 
$n(1)=1$. Conversely, let $(V,v,h)$ be a triple consisting of a 2-dimensional vector space $V$ over $L$, a non-degenerate hermitian form $h$ on $V$, and a point $v\in V$ such that 
$h(v)=1$. Then one can define a quaternion algebra $H_V$ such that $H_V=V$, as vector spaces over $K$,  $v$ is the identity element of $H_V$,  
 and $l \mapsto l v$ is an embedding of $L$ into $H_V$, where 
$l v $ is the scalar multiplication inherited from $V$. The reduced norm is, of course, equal to the hermitian form $h$. This discussion can be placed in a categorical context; 
the following two categories are isomorphic: 
 one whose objects are pairs $(i,L)$ where $i$ is an embedding of $L$ into a quaternion algebra, the other whose objects are triples $(V,v,h)$ of 
pointed, non-degenerate, binary hermitian spaces. 
We now observe that different choices of the point $v$ representing 1 give isomorphic objects. Indeed, if $u\in V$ is another element in $V$ representing 1, then $u$ is a unit element in the quaternion algebra 
$H_V$ arising from the triple $(V, v, h)$. Right multiplication by $u$ gives an isomorphism  of the triples $(V, v,h)$ and $(V, u, h)$. 
As a consequence, there is a bijection between isomorphism 
classes of embeddings of $L$ into quaternion algebras and non-degenerate binary hermitian spaces $(V,h)$ representing 1. 

As the next result, we establish an integral version of this bijection. More precisely, assume that $K$ is the field of fractions of a Dedekind domain $A$, and $B$  the integral 
closure of $A$ in $L$.  An order $O$ in $H$ is an $A$-lattice, containing the identity element and closed under the multiplication in $H$. An embedding $i: L \rightarrow H$ is 
called optimal if $i(L)\cap O = i(B)$. If this is the case, then $O$ is naturally a projective $B$-of rank $2$, such that the norm $n$ takes values in $A$ and represents 1. 
Conversely, let $\Lambda \subset V$ be a lattice, i.e. a projective $B$-module of rank 2, such that $h$ is integral on $\Lambda$, i.e. 
$h(\Lambda) \subseteq A$. Such a pair $(\Lambda, h)$ is called an integral binary hermitian form. 
Assume that there exists $v\in \Lambda$ such that $h(v)=1$. We show that $\Lambda$ is an order (denoted by $ O_{\Lambda}$) in $H_V$, and thus we have 
constructed a quaternionic order and an embedding of $B$ into it, given by $b\mapsto bv$. Just as in the field case, different choices of  $v$ give isomorphic objects and, as a consequence, 
there is a bijection between isomorphism 
classes of embeddings of $B$ into quaternionic orders and non-degenerate integral binary hermitian forms  $(\Lambda,h)$ representing 1. 

We now turn to the question when a hermitian form represents 1, which is especially interesting in the integral case. To that end we establish a relationship 
between invariants  of the objects studied. The determinant $d(\Lambda, h)$ of an integral binary hermitian form $(\Lambda, h)$ is a fractional ideal in $K$ generated by 
$\det(s(v_i, v_j))$ for all pairs $(v_1,v_2)$ of elements of $\Lambda$.  
We note that $d(\Lambda, h)$ is not necessarily an integral ideal. Integrality of $h$ implies only that $s$ takes values in the different ideal of $B$ over $A$ in $L$. 
 Since the discriminant $D_{B/A}$ is the norm of the inverse different, the product $ D_{B/A} \cdot d(\Lambda, h)$ is an integral ideal. This was 
 previously observed in \cite{savin-bestvina} where the discriminant of an integral binary hermitian form is defined as this product. Thus we follow the same convention 
 and define the discriminant of  $(\Lambda,h)$  to be the integral ideal  $\Delta(\Lambda, h) =D_{B/A} \cdot d(\Lambda, h)$. If $h$ represents 1, so $\Lambda$ is 
 the order $O_{\Lambda}$, then we have an identity 
\[ 
\Delta( O_{\Lambda})= \Delta(\Lambda, h) 
\] 
of integral ideals where $\Delta( O_{\Lambda})$ is the discriminant of the order $O_{\Lambda}$. 
If $K=\mathbb Q$ and $L$ an imaginary quadratic extension, then these invariants can be refined to be integers, and this is an equality of two integers. 
The following is the most interesting result in this paper: 


\begin{thm}  Assume $L$ is a tamely ramified, imaginary quadratic extension of $\mathbb Q$. Let $B$ be the maximal order in $L$. 
Let $\Delta<0$ be a square free integer. 
There is a bijection between the isomorphism classes of embeddings of $B$ into quaternionic orders of discriminant $\Delta$ and integral binary 
$B$-hermitian forms of discriminant $\Delta$. 

\end{thm} 

The key here is to show that an integral binary 
hermitian form $(\Lambda, h)$ of discriminant $\Delta$ represents 1. 
The condition $\Delta$ is square free assures that $h$ represents 1 on 
$\Lambda \otimes \mathbb Z_p$ for all $p$. The condition $\Delta <0$ implies that $h$ is indefinite on $\Lambda \otimes \mathbb R$. 
The theorem follows from the integral version of the Hasse-Minkowski theorem for indefinite 
quadratic forms in 4 variables.

\section{Hermitian Forms}\label{herm-form} 

Let $K$ be a commutative field of any characteristic. Let $L$ be a separable quadratic $K$-algebra. In particular, there exists an 
involution $l \mapsto \bar l$ on $L$ such that the set of fixed points is $K$. A vector space over $L$ is a free $L$-module. 
We go over some basic results about hermitian forms on vector spaces. 

\begin{definition}
 Let $V$ be a vector space over  $L$.  A \textbf{hermitian symmetric form} is a
function $s: V \times V \rightarrow L$ which is $L$-linear in the first
argument and $s(x,y) = \conj{s(y,x)}$.

\end{definition}

Let $h: V \rightarrow K$ be the function defined by $h(x)=s(x,x)$ for all $x\in V$. This is a quadratic form, if we consider $V$ a vector space over $K$. in particular, 
the function $b(x,y) = h(x + y) - h(x) - h(y)$ is $K$-bilinear. 
If we let $n_L(l) = l \conj{l}$, for $l\in L$, then  $h(lx) = n_L(l) h(x)$. 
These two facts are a sufficient characterization of hermitian forms for our
purposes, due to the

\begin{prop} \label{P:bilinear} Let $V$ be a vector space over $L$. 
Let $h: V \rightarrow K$ be a quadratic form, where $V$ is considered a vector space over $K$, satisfying $h(lx) = n_L(l) h(x)$
for all $l \in K$. Then there is a unique hermitian symmetric form $s(x,y)$ such that $h(x)
= s(x, x)$.
\end{prop}

\begin{proof}
Assume that $s$ exists. Then, for all $x, y \in V$ and $l \in L $,  
\begin{align*}
s(x,y) + s(y,x) & = b(x,y) \\
l s(x,y) + \conj{l } s(y,x) & = b(lx, y). 
\end{align*}
We view this as a $2\times 2$ system in unknowns $s(x,y)$ and $s(y,x)$. If $l\neq \bar l$ then this system has a unique solution given by 
\begin{equation}\label{herm-formula}
s_l(x,y) = \dfrac{\conj{l}b(x,y) - b(lx, y)}{\conj{l} - l}
\end{equation} 
where the subscript $l$ indicates that the solution, so far, depends on the choice of $l$. 
(Note that $\conj{l} - l$ is invertible, even when $L\cong K^2$.) We claim that $s_l(x,y)$ is independent of $l$. Indeed, let
$l = c + d \pi$, where $L=K[\pi]$, and $c,d\in K$. Then

\begin{align*}
s_l(x, y) & = \dfrac{\conj{l} b(x, y) - b(lx, y)}{\conj{l} - l} \\
& = \dfrac{c b(x, y) + d \conj{\pi} b(x, y) - b(cx + d \pi x, y)}{c + d
\conj{\pi} - c - d\pi} \\
& = \dfrac{d \conj{\pi} b(x, y) - d b(\pi x, y)}{d\conj{\pi} - d\pi}
\\
& = s_\pi(x, y).
\end{align*}
Hence $s_l = s_\pi$ and we can write $s$ for any of the $s_l$. It is then easy to check
that $s(x, x) = s_l(x, x) = h(x)$.

Now we check that $s(x,y)$ is linear in the first variable. It is clear that 
$s(x_1 + x_2, y) = s_l(x_1 + x_2, y)$. Let $l\in L \setminus K$. Then 
\begin{align*}
s(l x, y) & = s_{\conj{l}}(lx, y)  \\
& = \dfrac{l b(l x,y) - b(n_L(l) x, y)}{l - \conj
{l}} \\
& = \dfrac{l b(l x,y) - n_L(l) b(x, y)}{l- \conj{l}}
\\
& = l \dfrac{\conj{l} b(x, y) - b(lx,y)}{\conj{l} -l} \\
& = l s_l(x,y) = l s(x,y),
\end{align*}
which shows $s$ is linear in the first argument. 
It remains to check that $s(x, y) = \conj{s(y, x)}$. 
Note that $b(\conj{l}
x, \conj{l} y) = n_L(l) b(x, y)$. Replacing $x$ with $lx$, we find that $b(n_L(l)
x, \conj{l} y) = n_L(l) b(lx, y)$, hence $b(lx, y) = b(x, \conj{l} y)$. Then
\begin{align*}
s(y, x) = s_{\conj{l}}(y, x) & = \dfrac{l b(y, x) - b(\conj{l} y, x)}{l -\conj
{l}} \\
& = \dfrac{l b(x, y) - b(lx, y)}{l - \conj{l}} \\
& = \conj{s_l(x, y)} = \conj{s(x, y)},
\end{align*}
or $s(x, y) = \conj{s(y, x)}$, as desired. Uniqueness is clear. 
\end{proof}

Thus, we can make the following definition.

\begin{definition} Let $V$ be a vector space over $L$. 
A \textbf{hermitian form} is a quadratic form $h: V \rightarrow K$, where $V$ is a considered a vector space over $K$, such that $h(lv) = n_L(l) h(v)$ for all $l \in L$, and for all $v \in V$, 
\end{definition} 

\section{Quaternion algebras} 

Let $H$ be a quaternion algebra over $K$, i.e. a central simple algebra over $K$ of dimension $4$. 
We know that by general theory (e.g. \cite[\S 9a]{reiner}), every element $x\in H$ 
 satisfies a reduced characteristic polynomial of degree two,
associated to its action on the algebra via left-multiplication. The reduced trace $\tr(x)$ is
the sum of the roots of the reduced characteristic polynomial of $x$. Let the
reduced norm $n(x)$ be the product of the roots.
 The trace furnishes $H$ with an involution
$\conj{x} := \tr(x) - x$ such that $n(x)=x\conj{x}$.  Let $L$ be a separable quadratic field
extension of $K$ and $i: L \rightarrow H$ an embedding. Since $L$ is separable, the restriction of the involution of $H$ to $L$ is non-trivial. 
The norm $n$ is multiplicative, i.e. $n(xy)=n(x)n(y)$ for all $x,y\in H$. Hence  the norm is a hermitian
form on $H$. This hints at a correspondence between embeddings of $L$ into quaternion algebras 
and hermitian forms. The next section will formalize this idea. We finish this section with the following structural result.

\begin{prop}\label{quat-struct}
Let $L$ be a separable quadratic field
extension of $K$. If $i: L \rightarrow H$ is an embedding, then there exists $u \in H$ such that $ \bar u =-u$ and 
 $ \theta \in K^{\times}$ satisfying $H = L + Lu$, $u^2 = \theta$, and $um = \conj{m}u$
  for all $m \in L$ (cf. \cite[Definition 1.1, Chapter 1]{vigneras}). For every $v=x+yu \in  L + Lu$, 
  \[ 
  n(v)= n_L(x) - n_L(y)\theta, 
  \] 
  where $n_L: L \rightarrow K$ is the norm. 
\end{prop}
\begin{proof} By Proposition \ref{P:bilinear}, there exists a unique, $L$-hermitian bilinear form 
$s(v,w) : V \times V \rightarrow L$ such that $s(v,v)=n(v)$ for all $v\in H$. 
 Take $u\neq 0$ with $s(u, 1) = 0$. Then $H=L + Lu$.  For any
$l \in L \setminus K$, \[
0 = s(u, 1) = s_l(u, 1) = \dfrac{\conj{l}b(u, 1) - b(lu, 1)}{\conj{l} - l}.
\] Since $b(u, 1) = n(u + 1) - n(u) - n(1) = \tr(u)$ and $b(lu, 1) = \tr(lu)$, hence
$\conj{l} \tr(u) = \tr(lu)$. If $\tr(u) \neq 0$, then $\conj{l} \in K$, a
contradiction. Hence $\tr(u) = 0$ and $\tr(lu) = 0$. The first of these implies
that $\conj{u} = -u$, and so $0 = \tr(lu) = lu + \conj{u} \conj{l} = lu - u\conj
{l}$, i.e. that $lu = u \conj{l}$, or $u l= \conj{l} u$. Finally, since $\tr(u) =
0$, $u^2 = -n(u) \in K$. Simplicity of $H$ forces $\theta$ to be non-zero. The formula for $n(v)$ follows from orthogonality of 1 and $u$. 
\end{proof}

Note that the hermitian symmetric form $s$, attached to $n$, in the basis $1,u$ is given by a diagonal matrix with $1$ and $-\theta$ on the diagonal. In particular, 
invertibility of $\theta$ implies that the hermitian form is non-degenerate.

\section{Correspondence Over a Field}

Here $L$ continues to be a separable quadratic extension of $K$.
The correspondence alluded to before 
Proposition \ref{quat-struct} can be formalized as an isomorphism of the
categories \Herm\textsubscript{$L$}\ of non-degenerate pointed hermitian spaces and
\Quat\textsubscript{$L$} of embeddings of $L$ into quaternion algebras over $K$. 
We will define these terms more precisely, and give a proof of the isomorphism.

\begin{definition}

Let \Quat\textsubscript{$L$} be the category whose objects are pairs $(i, H)$
where $H$ is a quaternion algebra over $K$, and $i : L \rightarrow H$ is an embedding of $L$ into $H$. 
The morphisms $f : (i, H) \rightarrow (i', H')$ are morphisms $f : H \rightarrow H'$ of $K$-algebras
satisfying $f \circ i = i'$. Note that any $f$ must be an isomorphism of $H$ with
$H'$.

\end{definition}

\begin{definition}

Let \Herm\textsubscript{$L$} be the category whose objects are triples $(V, v,
h)$, called \textbf{pointed hermitian spaces}, where $V$ is a
two-dimensional $L$-vector space, $h: V \rightarrow K$ a non-degenerate
hermitian form, and $v \in V$ such that $h(v) = 1$. The morphisms from $(V,
v, h)$ to $(V', v', h')$ are isomorphisms $f: V \rightarrow V'$ with $f(v) = v'$
and $h(x) = h'(f(x))$.

\end{definition}

We can now define a functor $F:$ \Quat\textsubscript{$L$} $\rightarrow$
\Herm\textsubscript{$L$}, due to the following easy lemma: 
\begin{lem}
Let $F:$ \Quat\textsubscript{$L$} $\rightarrow$ \Herm\textsubscript{$L$} be
given by $F(i, H) = (H, 1, n)$, where $n$ is the norm form of $H$, and $1$
is the identity element of $H$. Given a morphism $f: (i, H) \rightarrow (i', H')$,
define $F(f) := f$. Then $F$ is a functor.
\end{lem}

Now we construct a functor $G:$ \Herm\textsubscript{$L$} $\rightarrow$
\Quat\textsubscript{$L$} by defining a multiplication $\cdot$ on a given $(V, v, h)$,
motivated by Proposition~\ref{quat-struct}. Recall that there is a unique
hermitian symmetric form $s (v,w)$ on $V$ with $h(x) = s(x, x)$ for all $x\in V$.   
Embed $L$ into $V$ via $l \mapsto lv$ where $lv$ denotes the scalar multiplication. 
\begin{enumerate}
\item $v$ is the identity element, 
\item for $x \in V$,  and $l\in L$, $l \cdot x := lx$, the scalar multiplication, 
\item for $y \in L^\perp = \{ w \ | \ s(v, w) = 0\}$, $y \cdot (lv) := \conj{l}y$,
for $l \in L$.
\item for $y \in L^\perp$, $y \cdot y := -h(y)v$.
\item conjugation on $V$ is $\conj{lv+y}=\conj{l}v -y$,  for $l\in L$ and $y \in L^\perp$,
\end{enumerate}
 This gives $V = L \oplus L^\perp$ a structure of quaternion algebra, as in Proposition \ref{quat-struct}, 
 since we can choose $u$ to be any element of $L^\perp$, and $\theta = -h(u)$.
  Non-degeneracy of $s$ implies that $\theta\neq 0$. The norm form is $h$. 
Let $H_V$ denote this algebra. Then we have  

\begin{lem}
Define $G(V, v, h) := (i_v, H_V)$, where $i_v(l)=lv$ for $l\in L$, and for a morphism $f: (V,
v, h) \rightarrow (V', v', h')$, define $G(f) := f$. Then $G$ is a functor from 
\Herm\textsubscript{$L$} to \Quat\textsubscript{$L$}.
\end{lem}

It is clear that the functors $F$ and $G$ are inverses of each other. Thus, we have

\begin{thm}
The categories \Quat\textsubscript{$L$} and 
\Herm\textsubscript{$L$} are isomorphic. 
\end{thm}

Let $(V,v, h)$ be a pointed binary hermitian space and $H_V$ the quaternion algebra with the identity element $v$ arising from it. 
Let $u\in V$ be another element in 
$V$ such that $h(u)=1$. Then $v,u$, understood as elements in $H_V$,  satisfy $v\cdot u=u$. Thus, the right multiplication by $u$ gives an isomorphism of the
pointed symmetric spaces $(V,v,h)$ and $(V, u, h)$.  This shows that the isomorphism classes of objects in the category \Herm\textsubscript{$L$}  do not depend 
on the choice of the point representing 1.  Thus we have the following:

\begin{cor} \label{C:field_equiv} 
The functors $F$ and $G$ give a bijection between the isomorphism classes of embeddings of $L$ into quaternion algebras and the isomorphism classes 
of non-degenerate binary hermitian spaces $(V,h)$ such that $h$ represents 1. 
\end{cor} 

\section{Correspondence Over a Dedekind Domain}

Assume that $K$ is a field of fractions of a 
 Dedekind domain $A$. Let $B$ be the integral closure of $A$ in $L$. 
 It turns out an 
integral version of the above isomorphism can be established, replacing the field
with a Dedekind domain, quaternion algebras with orders, and binary hermitian
forms on a vector space with integral binary hermitian forms on a $B$-projective
module of rank 2. Let $i: L \rightarrow H$, be an embedding of $L$ into $H$. 
Recall that a lattice $\Lambda$ in the quaternion algebra $H$ is a finitely generated $A$-submodule containing 
a $K$-basis of $H$.  If the lattice $\Lambda$ is a $B$-module, then it is also projective of rank 2. 
An order in $H$  is a lattice containing the identity element and closed under the multiplication.

\begin{definition}

Let \Ord\textsubscript{$B$} be the category whose objects are pairs $(i,
\mathcal{O})$ where $i : B \rightarrow \roi{}$ is an embedding into an order
$\roi{}$ of quaternion algebra. The morphisms from  $(i, \mathcal{O})$ to 
$ (i', \mathcal{O}')$  are isomorphisms $f : \mathcal{O}
\rightarrow \mathcal{O}'$ of $A$-algebras satisfying $f \circ i = i'$.

\end{definition}

\begin{definition} A binary $B$-hermitian form is a pair $(\Lambda, h)$ where $\Lambda$ is a 
projective $B$-module of rank two, and $h: \Lambda \otimes_A
K \rightarrow K$ is an $L$-hermitian form. The form is integral if $h(\Lambda) \subset A$. 

\end{definition}

In the above definition we have identified $\Lambda$ with $\Lambda\otimes 1 \subset V = \Lambda \otimes_A K$.  

\begin{definition}

Let \Proj\textsubscript{$B$} be the category whose objects are triples $(\Lambda,v, h)$, 
called \textbf{pointed integral hermitian spaces}, where $(\Lambda, h)$ is an integral, binary $B$-hermitian form, with $h$ non-degenerate, 
 and $v \in \Lambda$ with $h(v) = 1$. The morphisms from $
(\Lambda, v, h)$ to $ (\Lambda', v', h')$ are $B$-module isomorphisms $f: \Lambda
\rightarrow \Lambda'$ preserving the pointed hermitian structure, i.e. for $x \in
\Lambda$, $h(x) = h'(f(x))$ and $f(v) = v'$.

\end{definition}

We will prove that the two categories are isomorphic. In one direction we have 
a functor given by the following: 

\begin{lem}
Let $P:$ \Ord\textsubscript{$B$} $\rightarrow$ \Proj\textsubscript
{$B$} be defined by $P(i, \mathcal{O}) := (\mathcal{O}, 1, n)$, where $n$ is
the norm of the quaternion algebra $\roi{} \otimes_A K$, and for a morphism
$f: (i, \mathcal{O}) \rightarrow (i', \mathcal{O}')$, $P(f) := f$. Then $P$ is a
functor between these two categories.
\end{lem}

In order to define a functor $Q:$ \Proj\textsubscript{$B$} $\rightarrow$
\Ord\textsubscript{$B$}, we need to define a multiplication on a pointed
$B$-hermitian module $(\Lambda, v, h)$. On $V = \Lambda \otimes_A K$, as in the previous section, we have a 
quaternion algebra structure  $H_V$ arising from $h$ and the 
associated $L$-hermitian bilinear form $s: V \times V \rightarrow L$, such that $v$ is the identity. 
We need the following: 

\begin{lem}
The $B$-module $\Lambda$, regarded as a
subset of $H_V$ (recall $V = H_V$ as sets), is closed under the multiplication on
$H_V$ i.e. it is an order in $H_V$ denoted by $\mathcal O_{\Lambda}$ 
\end{lem}
\begin{proof}
We first reduce to the case of $A$ and $B$ principal ideal domains. 
Let $\prim$ be any prime of $A$, $S_\prim = A - \prim$, and $A_\prim = S_\prim^{-1} A$,
$B_\prim = S_\prim^{-1} B$ the localizations of $A$ and $B$ respectively. 
 Then $B_\prim$ is a semi-local Dedekind domain, hence it is a principal ideal domain. 
Let $\Lambda_\prim = S_\prim^{-1} \Lambda$. If we show, for $x,
y \in \Lambda$ , that $x \cdot y \in \Lambda_\prim$, for all $\prim$, then using $\Lambda =
\bigcap_\prim \Lambda_\prim$, we  conclude that $x \cdot y \in \Lambda$. 

Hence we can assume that $B$ is a principal ideal domain. Then $\Lambda$ is a free $B$-module of rank $2$. Note that
$v$ is a primitive element of the lattice $\Lambda$ since $h(v) = 1$. Hence $v$  is  member of a $B$-basis $(v,w)$ of $\Lambda$ 
 \cite[Ch. 7, Theorem 3.1, page 106]{cassels}.  Let  $v^\perp = w - s(w, v) v$, then $s(v, v^\perp)= 0$.
Let $\gamma = s(w, v)$ and $\theta = v^\perp \cdot v^\perp = -h(v^\perp)$.
Then to check $x \cdot y \in \Lambda$, we just need to check that
\begin{align*}
w \cdot w = (v^\perp + \gamma v)\cdot (v^\perp + \gamma v)
\end{align*}
is an element of $\Lambda_\prim$. But we can compute
\begin{align*}
w \cdot w & = (v^\perp + \gamma \cdot v) \cdot (v^\perp + \gamma \cdot v) \\
& = v^\perp \cdot v^\perp + v^\perp \cdot \gamma \cdot v + \gamma \cdot v \cdot
v^\perp + \gamma \cdot v \cdot \gamma \cdot v \\
& = \theta \cdot v + \conj{\gamma} \cdot v^\perp + \gamma \cdot v^\perp +
\gamma^2 \cdot v \\
& = (\theta + \gamma^2) v + (\gamma + \conj{\gamma}) v^\perp \\
& = (\theta + \gamma^2 - \tr(\gamma)\gamma) v + \tr(\gamma) w \\
& = (\theta - n(\gamma)) v + \tr(\gamma) w
\end{align*}
Now note that $n(\gamma) = \gamma \conj{\gamma}= h(\gamma \cdot v)$. Hence,
\begin{align*}
-\theta = s(v^\perp, v^\perp) & = s(w, w) + s(-\gamma v, w) + s(w, -\gamma v) + 
s(-\gamma v, -\gamma v) \\
& = h(w) + h(\gamma v) - \gamma s(v, w) - \conj{\gamma} s(w, v) \\
& = h(w) + n(\gamma) - 2 n(\gamma) = h(w) - n(\gamma),
\end{align*}
since $s(v,w) = \conj{\gamma}$. Hence $\theta - n(\gamma) = -h(w)$,
which is in $A$. Finally,
\begin{align*}
\tr(\gamma) = \gamma + \conj{\gamma} & = s(w, v) + s(v, w) \\
& = s(v + w, v + w) - s(v, v) - s(w, w) \\
& = h(v + w) - h(v) - h(w),
\end{align*}
which is also in $A$. Thus, the product from $V$ preserves $\Lambda$. 
\end{proof}

 Then we have
the lemma
\begin{lem}
Let $Q:$ \Proj\textsubscript{$B$} $\rightarrow$ \Ord\textsubscript{$B$} be
defined by $Q(\Lambda, v, h) = (i_v, \roi\Lambda)$,  $i_v(b)=b v$ for  $b\in B$, and 
for a morphism $f : (\Lambda, v, h) \rightarrow (\Lambda', v',h')$, define $Q(f) := f$. Then $Q$ is a functor.
\end{lem}

Again, it is easy to check that the functors $P$ and $Q$ are inverses of each other. In particular: 
\begin{thm}\label{integral-correspondence}
The categories 
\Ord\textsubscript{$B$} and  \Proj\textsubscript
{$B$} are isomorphic. 
\end{thm}

Arguing in the same way as in the case of fields, Corollary \ref{C:field_equiv}, 
 two pointed binary integral $B$-hermitian modules $(\Lambda, v, h)$ and $(\Lambda, u, h)$ are isomorphic. Thus the isomorphism classes 
in \Proj\textsubscript{$B$} are the same as the isomorphism classes of binary integral hermitian spaces $(\Lambda, h)$ such that $h$ represents 1. 

\begin{cor} \label{C:int_equiv} 
The functors $P$ and $Q$ give a bijection between the isomorphism classes of embeddings of $B$ into quaternionic orders of and the isomorphism classes 
of binary integral $B$-hermitian spaces $(\Lambda,h)$ such that $h$ represents 1. 
\end{cor} 

All that remains is to examine in which situations there is  existence of a
point with $h(v) = 1$. This can be done, in the field and domain cases, by using
local-global principles.

\section{Discriminant Relation}\label{norm}

We are in the setting of the previous section. 
\begin{lem} Let $\Lambda$ be a lattice in $H$. The fractional ideal $d(\Lambda)$ in $K$ generated by 
$\det(\tr(u_i u_j))$, for all quadruples  $(u_1, \ldots , u_4)$ of elements in $\Lambda$ is a 
square. 
\end{lem} 
\begin{proof}  
Assume firstly that $A$ is a principal ideal domain. Then $d(\Lambda)$ is a principal ideal generated by $\det(\tr(u_i u_j))$ where 
 $(u_1, \ldots , u_4)$ is any $A$-basis  of $\Lambda$. Let $\Lambda'$ be another lattice and $(u'_1, \ldots , u'_4)$ its basis. 
 Let $T: H \rightarrow H$ be the linear transformation such that $T(u_i)= u_i'$ for all $i$. Then 
 \[ 
 d(\Lambda')= \det(T)^2 d(\Lambda). 
 \] 
 Hence, in order to show that $d(\Lambda)$ is a square, it suffices to do so for one lattice. Write $H=L + L u$, as in Proposition \ref{quat-struct}.  Let 
 $\Lambda= B + Bu$. Assume that $B= A[\pi]$ where $\pi$ is a root of the polynomial $x^2 + ax +b$. Pick $1, \pi, u, \pi u$ as a basis of $\Lambda$. Then the matrix of traces is 
 \[\left(\begin{matrix}
    2 & -a & 0 & 0 \\
    -a & a^2 - 2b & 0 & 0 \\
    0 & 0 & -2\theta & a\theta \\
    0 & 0 & a\theta & -2 b \theta \\
  \end{matrix}\right),\]
  and its determinant is $(a^2 - 4b)(4b - a^2) \theta^2 = -(a^2-4b)^2 \theta^2$. 
  Then $d(\Lambda)= (D_{B/A}\theta)^2$ (as ideals) where $D_{B/A}$ is the discriminant of $B$ over $A$. The general case 
 is now treated using localization. Let $\prim$ be a maximal ideal in $A$. Then, one easily checks, 
 $d(\Lambda)_{\mathfrak p}= d(\Lambda_{\mathfrak p})$. Since $A_{\mathfrak p}$ is a local Dedekind ring, hence a principal domain, $d(\Lambda_{\mathfrak p})$ must be 
 an even power of $\mathfrak p$. 
\end{proof} 

\begin{definition} 
The discriminant of an $A$-lattice $\Lambda$ in $H$ is the fractional ideal $\Delta(\Lambda)$ in $K$ such that $\Delta(\Lambda)^2=d(\Lambda)$. 
\end{definition} 

If $K=\mathbb Q$ then we can refine the notion of $\Delta(\Lambda)$ to be a rational number, the generator of this ideal. 
We pick $\Delta(\Lambda)$ to be the positive generator if 
$H\otimes_{\mathbb Q} \mathbb R$ is the matrix algebra, and negative otherwise.

\begin{definition} Let $\Lambda$ be a projective $B$-module of rank 2 and $h$ a hermitian form on $\Lambda$, not necessarily integral. 
 Let $s$ be the bilinear hermitian form on $\Lambda$ such that 
$s(x,x)=h(x)$, for all $x\in \Lambda$. Let $d(\Lambda, h)$ be the fractional ideal in $K$ generated by 
 $\det(s(v_i,v_j))$ for all pairs  $(v_1, v_2)$  of elements in $\Lambda$. 
\end{definition} 

If $K=\mathbb Q$ and $L$ a complex quadratic extension we refine $d(\Lambda,h)$ to be the positive generator of this ideal if $h$ is a positive  definite 
hermitian form on $\Lambda\otimes_{\mathbb Z} \mathbb R$ and negative otherwise. In other words this sign is the sign of $\det(s(v_i,v_j))$, for 
any basis $(v_1, v_2)$  of $\Lambda \otimes_{\mathbb Z} \mathbb R$.

\begin{thm} \label{T:discr_rel} 
 Let $\Lambda$ be a lattice in $H$ that is also a $B$-module. Then we have the following identity of fractional ideals in $K$: 
 \[ 
 \Delta(\Lambda)= D_{B/A} \cdot d(\Lambda, n). 
 \] 
 If $K=\mathbb Q$ and $L$ a complex quadratic extension, this is an identity of rational numbers. 
 
\end{thm} 
\begin{proof} 
This is an identity of two fractional ideals, so it can be checked by localization at every prime ideal $\prim$. 
 The localizations $A_{\mathfrak p}$ and $B_{\mathfrak p}$ are a local and a semi-local, respectively, Dedekind domains. 
Thus they are principal ideal domains \cite{serre}.  Hence, by localization, the proof can be reduced to the case of principal ideal domains. So we shall assume that 
$A$ and $B$ are both principal ideal domains. In that case $d(\Lambda, n)$ is a principal ideal generated by $\det(s(v_i,v_j))$ where $v_1, v_2$ is a $B$-basis 
of $\Lambda$.  Let $\Lambda'$ be another lattice in $H$ that is also a $B$-module. Let $v_1', v_2'$ be a $B$-basis of $\Lambda'$. 
Let $T: H \rightarrow H $ be the $L$-linear map defined by $T(v_i)=v'_i$. Then 
 \[ 
 d(\Lambda',n)= n_L({\det}_L(T))d (\Lambda, n). 
 \] 
Now note that if we view $T$ as a $K$-linear map, then $\det_K(T)= n_L(\det_L(T))$. Since $\Delta(\Lambda')= \det_K(T)\Delta(\Lambda)$, it follows that both sides of 
the proposed identity change in the same way, if we change the lattice. Hence it suffices to check the identity for one lattice. Take $\Lambda = B + B u$, then both sides 
are ideals generated by $D_{B/A} \theta$,  where $-\theta= u^2$, so the identity of ideals holds. 
The last statement is an easy check of signs, since the discriminant of complex quadratic fields is negative. 

\end{proof} 
We finish this section with the following proposition which defines the discriminant of an integral binary hermitian form, and establishes its integrality. 

\begin{proposition} \label{P:integrality_disc} 
Let $\Lambda$ be a projective $B$-module of rank 2 and $h$ an integral hermitian form on $\Lambda$. Then 
$\Delta (\Lambda, h)= D_{B/A} \cdot d(\Lambda , h)$ is an integral ideal, called the discriminant of $(\Lambda, h)$. 
\end{proposition} 
\begin{proof} By localization, we can assume that $A$ is local. Let $v_1,v_2\in \Lambda$ and consider the entries of the matrix $(s(v_i,v_i))$. Since 
$s(v,v)=h(v)$ the diagonal entries are integral. However, the off-diagonal entries need not be integral. Indeed, for every $l\in B$, such that $l\neq \conj l$, 
we have the formula 
\[ 
s(v_1, v_2) = \dfrac{\conj{l}b(v_1,v_2) - b(lv_1, v_2)}{\conj{l} - l}. 
\] 
If $\prim$ splits or is inert in $B$, then there exists $l$ such that $\conj l - l$ is a unit. Hence $\det(s(v_i,v_j))$ is in $A$. If $\prim$ ramifies, then we can take 
$l$ to be the uniformizing element in $B$. In this case  $B=A[l ]$ and the discriminant $D_{B/A}$ is precisely the norm of $\conj l -l$. It follows that 
$D_{B/A} \cdot \det(s(v_i,v_j))$ is integral. 
 \end{proof}

\section{Representing one} 
Here we address the question when a binary hermitian form
represents 1, integral or otherwise. We first address the field case. The case of
integral representations is significantly more complicated; a local-global
principle is provided by \cite[Theorem 104:3]{omeara}, which will reduce the
problem of an integral representation of 1 to the problem of integral
representations over $p$-adic rings.

\subsection{Over a Field} Assume that $K$ is a number field. 
 Any binary $L$-hermitian form $h$ is a quaternary
quadratic form over $K$. Recall that $h$ is positive semi-definite at a real 
place of $K$ (i.e. an embedding of $K$ into $\mathbb R$) if $h$ is always greater than or equal to 0 on $V\otimes_K\mathbb R$. We say that $h$ is 
positive semi-definite or indefinite if it is positive semi-definite or
indefinite at all real places of $K$. Then we have the following

\begin{lem}
The quadratic form $h$ is positive semi-definite or indefinite if and only if
$h$ represents 1. 
\end{lem}
\begin{proof} 
Consider the 5-dimensional space 
$V \oplus K$ with the quadratic form $f(v,x)=h(v)-x^2$. Since $V\oplus K$ is 5-dimensional, by the Hasse-Minkowski theorem, it represents 0 if and only if it represents 0 at all 
real places. If $h$ is positive semi-definite or indefinite, the form $f$ is 
semi-define at all real places, so it represents 0 at all real places. Hence $f$ 
represents 0, and if $x \neq 0$ in the representation, then $h$ represents 1 by dividing by $x$.
Otherwise, $h$ represents 0, and so it represents any element of $K$. Conversely,
if it represents 1, then it does so at every real place. Hence it cannot be
negative semi-definite, and thus, it must be positive semi-definite or
indefinite.
\end{proof}

\subsection{Over integers} 
We assume that $K=\mathbb Q$ and $L$ a quadratic imaginary field. Thus, $A=\mathbb Z$ and $B$ is the ring of integers in $L$.  Let $D=D_{B/A}$ be the discriminant of 
$L$. We shall assume that $L$ is unramified at 2, in order to avoid discussing the difficult $2$-adic theory of integral hermitian and quadratic forms.

\begin{lem} Let $\Delta<0$ be a square-free integer. Let $\Lambda$ be a projective $B$-module of rank $2$ and 
$h$ an integral hermitian form on $\Lambda$ of the discriminant $\Delta$. Then there exists $v\in \Lambda$ such that $h(v)=1$ 

\end{lem} 
\begin{proof}  Let $\mathbb Z_p$ be the $p$-adic completion of $\mathbb Z$. Let $B_p=B\otimes \mathbb Z_p$ and $\Lambda_p=\Lambda\otimes \mathbb Z_p$. 
Note that the subscript $p$ denotes the completion and not localization in this proof. 
Since $h$ is assumed to be indefinite, 
by \cite[Ch. 9, Theorem 1.5, page 131]{cassels}, it suffices to 
  prove that for every $p$ there exists $v\in \Lambda_p$ such that $h(v)=1$.
  Assume that $p$ does not divide $D$ i.e. $p$ does not ramify in $L$. 
 Recall  that  $\val_p(\Delta) =0$ or $1$, by the assumption on $\Delta$. 
  We claim that there exists $u\in \Lambda$ such that $h(u)=x$ where $\val_p(x) =0$. Indeed, 
 if not,  then $v\mapsto h(v)/p$ is an integral form.  But the discriminant of $h/p$ is $\Delta/p^2$, not an integer, contradicting 
   Proposition \ref{P:integrality_disc}.  Furthermore, since $p$ does not ramify in $L$,  
 the norm map $n_L: B_p^{\times}\rightarrow \mathbb Z_p^{\times}$ is surjective. Hence we can rescale $u$ by an element in 
 $B_p^{\times}$  to get $v$ such that $h(v)=1$.

 Now assume $p$ divides $D$, so $p$ is tamely ramified in $L$. In particular, $p$ is odd, and we can write 
 $B_p=\mathbb Z[\pi]$ where $\pi$ is a uniformizer of $B_p$. We can assume that $\conj \pi =-\pi$ and, 
 for the price of changing the uniformizer in $\mathbb Z_p$, that $\pi \conj \pi=p$. 
  Pick a $B_p$-basis 
 $(v_1, v_2)$ of $\Lambda_p$ and write any $v\in \Lambda_p$ as a sum $v=xv_1+ yv_2$ where $x,y\in B_p$. Then 
 \[
h(v) = \left(
  \begin{matrix}
    x & y
  \end{matrix}\right)\left(
  \begin{matrix}
    \alpha & \gamma \\
    \overline{\gamma} & \beta
  \end{matrix}\right)\left(
  \begin{matrix}
    \overline{x} \\ \overline{y}
  \end{matrix}\right)
  \] 
  where $\alpha,\beta\in \mathbb Z_p$. 
By the assumption, we have  $\val_p(\Delta) =0$ or $1$.  If $\val_p(\Delta)=1$ then $\val_{\pi} (\gamma)\geq 0$, and it is easy to see that the matrix can be diagonalized, by 
a change of $B_p$-basis in $\Lambda_p$. Hence, in this case, 
the quadratic form $h$ on $\Lambda_p$ is equivalent to a diagonal form 
 \[ 
  a_1x_1^2+ a_2 x_2^2 + a_3 px_3^2 + a_4 px_4^2
 \] 
 for some  $a_i$  in $\mathbb Z_p^{\times}$. Assume now that $\val_p(\Delta)=0$. Then $\gamma = (a+b\pi)/\pi$, for some $a\in \mathbb Z_p^{\times}$ and 
 $b\in \mathbb Z_p$.  Writing $x = x_1 + x_2\pi$ and $ y = y_1 + y_2 \pi$,  the $4\times 4$ matrix representing the quadratic form $h$ in the variables 
 $(x_1, x_2, y_1, y_2)$ is 
 \[
 \left(\begin{matrix}
\alpha & 0 & b/2 & a/2 \\
0 & \alpha p & a/2 & p b/2 \\
b/2 & a/2 & \beta & 0 \\
a/2 & p b/2 & 0 & \beta p
\end{matrix}\right).
\]
The determinant of this matrix modulo $p$ is $a^4/16 \in \mathbb Z_p^{\times}$. 
Recall that, since $p$ is odd, every $\mathbb Z_p$-integral form can be diagonalized. Since 
  the determinant is $p$-adic unit, the form is 
 integrally equivalent to a diagonal form 
  \[ 
 a_1x_1^2+ a_2 x_2^2 + a_3 x_3^2 + a_4 x_4^2 
 \] 
for some  $a_i$  in $\mathbb Z_p^{\times}$. But in each of the two cases the diagonal quartic, quadratic form 
 represents 1 by \cite[Ch. 8, Theorem 3.1, page 115]{cassels}. This completes the proof of lemma. 
\end{proof} 

The following is a combination of the previous lemma, Corollary \ref{C:int_equiv} and Theorem \ref{T:discr_rel}. 
We remind the reader that an integral binary $B$-hermitian form is a $B$-hermitian 
form $h$ on a projective $B$-module $\Lambda$ of rank 2, such that $h(\Lambda) \subseteq \mathbb Z$. 

\begin{thm} \label{integral-reps}
Let $L$ be a quadratic imaginary extension of $\mathbb Q$ of odd discriminant $D <0$. Let $B$ be its maximal order. Let $\Delta<0$ be  
 a square free integer. Then there is a natural bijection between isomorphism classes of embeddings of $B$ into quaternionic orders 
of the discriminant $\Delta$ and isomorphism classes of integral binary $B$-hermitian forms of the discriminant $\Delta$. 

\end{thm} 

\section{Acknowledgements} 
We would like to thank Tonghai Yang for suggesting the relationship between optimal embeddings and integral binary hermitian forms, and to Dick Gross for a conversation on this subject. 
The first author has been supported by an NSF grant DMS-1359774. The second author was supported by REU and UROP funding at the University of Utah.

\bibliographystyle{amsplain}
\bibliography{refs}
\end{document}